\documentclass[12pt]{article}
\usepackage{pslatex}
\usepackage{amsmath}
\usepackage{amssymb}
\usepackage{amsthm}
\usepackage{graphicx}
\usepackage[caption=false]{subfig}
\usepackage{ulem}
\usepackage{algorithm}
\usepackage{algpseudocode}
\usepackage{multirow}
\usepackage[normalsize,bf]{caption}

\newtheorem{definition}{Definition}
\newtheorem{theorem}{Theorem}
\newtheorem{lemma}{Lemma}

\newtheorem{remark}{Remark}

\newcommand{\colzero}[1]{\fbox{#1}}
\newcommand{\colone}[1]{$\bullet$ #1}

\begin{document}

\title{Finite Projective Geometry based Fast, Conflict-free Parallel Matrix Computations}

\author{\footnotesize{Shreeniwas Sapre$^{\rm b}$ ~~~~~~
                      Hrishikesh Sharma$^{\rm a}$ ~~~~~~
                      Abhishek Patil$^{\rm a}$ ~~~~~~
                      B. S. Adiga$^{\rm c}$ ~~~~~~
                      Sachin Patkar$^{\rm a}$} \\
    \footnotesize{$^{\rm a}${\em Dept. of Electrical Engg, Indian Institute
    of Technology Bombay, India;}} \\
    \footnotesize{$^{\rm b}${\em Dept. of Computer Science, Indian Institute of
    Technology Bombay, India; and}} \\
    \footnotesize{$^{\rm c}${\em Tata Consultancy Services, Bangalore, India}}
}

\maketitle

\begin{abstract}
Matrix computations, especially iterative PDE solving (and the sparse
matrix vector multiplication subproblem within) using conjugate gradient
algorithm, and LU/Cholesky decomposition for solving system of linear
equations, form the kernel of many applications, such as circuit simulators,
computational fluid dynamics or structural analysis etc. The problem
of designing approaches for parallelizing these computations, to get
good speedups as much as possible as per Amdahl's law,
has been continuously researched upon. In this paper, we discuss approaches
based on thea use of finite projective geometry graphs for these two
problems. For the problem of conjugate gradient algorithm, the approach
looks at an alternative data distribution based on projective-geometry
concepts. It is \textit{proved} that this data distribution is an
\textbf{optimal} data distribution for scheduling the main problem of dense
matrix-vector multiplication. For the problem of parallel LU/Cholesky
decomposition of general matrices, the approach is motivated by the
recently published scheme for interconnects of distributed systems, perfect
difference networks. We find  that projective-geometry based graphs indeed
offer an exciting way of parallelizing these computations, and in fact many
others. Moreover, their applications ranges from architectural ones
(interconnect choice) to algorithmic ones (data distributions).
\end{abstract}


\section{Introduction}
\label{intro}

Computations related to large matrices arise in many applications, from
circuit simulators to well-known page ranking algorithm employed by Google.
An omnipresent type of matrix computation is linear system direct solver,
which uses LU or Cholesky factorization as its \textit{computational core}.
\textbf{LU decomposition} of a matrix \textbf{A} is its decomposition into
a lower triangular matrix and an upper triangular matrix, i.e. \textbf{A =
L$\cdot$U}. \textbf{Cholesky decomposition} decomposes a \textit{symmetric
positive definite matrix} \textbf{A} as {\textbf{A} =
\textbf{L}$\cdot$$\mathbf{L^{T}}$}. Cholesky algorithm is numerically
robust, and is often applicable due to the abundance of symmetric positive
definite linear systems in scientific computing domain. In particular,
Cholesky algorithm does not require pivoting, whereas LU decomposition
typically does require pivoting.  As the matrix sizes involved in
applications such as circuit simulators can be in order of millions or even
higher, expectedly, there has been plenty of research on parallelizing such
algorithms\,\cite{scalapack2}.

When the matrix \textbf{A} in the linear system \textbf{Ax}=\textbf{b} is
large but sparse, iterative methods are preferred over above methods.  For
\textit{symmetric positive definite} matrices, the method of
\textbf{conjugate gradients}\,\cite{SHEWCHUK94A} is often used. A large
variety of scientific and engineering applications solve system of partial
differential equations (PDEs) as well, using iterative solvers based on
this method.  The \textit{computational core} within this algorithm is a
multiplication of a sparse matrix with a dense vector(SpMV).  Hence, as per
Amdahl's law, best possible speedups can only be achieved when efforts are
directed to parallelize these computational cores, SpMV and LU/Cholesky
decomposition.

As such, parallelization of both matrix computations leads to
interconnecting memory blocks and processing units, in the usual sense.
Also, parallelization of the SpMV kernel entails problems such as load
balancing, and the relatively low ratio of computation to communication.
These are typical problems encountered during parallelization of
\textit{any} computational problem anyway\,\cite{vipin}. To address these
problems, a novel interconnection pattern was proposed by Karmarkar based
on \textit{finite projective geometries}\,\cite{karm1}. The processors and
memories are associated with elements of these geometries and the
interconnections are based on their incidence relations. The
\textit{computations assigned} to a processor, and \textit{corresponding
data distribution}, also depend on the geometry and incidence relations.
Because the geometry is symmetric in nature, \textit{the computational load
on each processor is balanced}. In fact, the load on memory blocks and
interconnect is also balanced. The automorphisms governing these geometries
are used to develop \underline{perfect-access patterns} and
\underline{perfect-access sequences}, which ensure that all the processors
and memories are simultaneously involved in \textit{communication of data
without any conflicts}. Algorithms to solve various problems on this
architecture can be developed using these properties.

For CG Algorithm, a lot of research work is actively being carried out on
the SpMV kernel implementations.  Typically one uses preconditioned CG
(PCG) as the workhorse algorithm for solving a linear system defined by
large sparse symmetric positive definite matrix.  The projective geometry
based interconnection pattern proposed by Karmarkar aims at improving
communication efficiency by \textit{superior utilization of communication
bandwidth} using perfect-access patterns and perfect-access sequences.  In
the solution to PCG, projective geometry \textit{further improves} this
efficiency, but in an \textit{independent} dimension.  We
use the finite projective geometries introduced in \cite{karm1}
as a basis to define
a novel data {distribution}, the (finite) projective data distribution.
The projective data distribution \textit{reduces the communication load} of
the algorithm, and this distribution is provably communication optimal.
The software prototype developed by us and also described in the first half
of this paper compares the performance using conventional data distribution
and the projective data distribution. The optimal projective data
distribution is patent pending \cite{pg_dd_patent}.

A naive PCG (and SpMV) implementation faces many potential performance
bottlenecks. These bottlenecks dominate the performance of the algorithm,
and hence their elimination is a pre-requisite for the effect of the data
distribution on the performance to manifest itself. Since these bottlenecks
and their solutions are independent of the main focus -- influence of
Projective Geometry principles -- of this paper, these solutions are only
outlined in this paper.

A lot of research on parallelization of LU/Cholesky decomposition
has been for dense matrices. Inter-processor data communication for
such matrices has been quite a challenging issue. Here, we present
our results for communication-efficient parallelization of this problem.
The schemes are based on \textit{more practical assumptions} than
used in scheduling ideas suggested in\,\cite{karm1}, and are motivated
by recent works\cite{rakov1,rakov2}. The choice of underlying graph
is a 4-dimensional projective space, justified later. The scheduling
schemes are based on indirect incidences based on subsumption relations
between projective subspaces. They are evaluated based on the amount
of communication and computation required in each of them. A comparison
is also drawn with conventional parallel architectures such as mesh.
We report these schemes in the latter half of this paper. As such,
we have found even more applications of projective geometry based
graphs in other areas, \textit{most notably in error correction coding}
and digital system design, that have been reported
separately\,\cite{expanders}, \cite{ldpc_pap}, \cite{cacs_pap},
\cite{ijpds_pap}.

\section{Projective Spaces based Interconnect Topologies}
\label{pg_details}

We first provide an overview of the fundamental
concepts of projective spaces, that have been used \textit{throughout}
our work. Projective spaces and their lattices are built using vector
subspaces of the bijectively corresponding vector space, one dimension
high, and their subsumption relations. Vector spaces being extension
fields, Galois fields are used to practically construct projective
spaces\,\cite{shulin}. Consider a finite field $\mathbb{F}$ = $\mathbb{GF}$(s)
with \textit{s} elements, where \textit{s} = $p^{k}$: k = +ve integer.

An example \textit{Finite Field} can be generated as follows. For
each value of $s$ in $\mathbb{GF}$(s), one needs to first find a
\textit{primitive polynomial} for the field. Such polynomials are
well-tabulated in various literature. For example, for the (smallest)
projective geometry, $\mathbb{GF}$($2^3$) is used for generation. One
primitive polynomial for this Finite Field is
$(x^3+x+1)$. Powers of the root of this polynomial, $x$, are then
successively taken, ($2^3$ -1) times, modulo this polynomial, modulo 2.
This means, $x^3$ is substituted with $(x+1)$, wherever required, since
over base field $\mathbb{GF}$(2), -1 = 1. A \textit{sequence} of such evaluations
lead to generation of the sequence of $(s-1)$ Finite field elements,
\textbf{other than 0}. Thus, the sequence of $2^3$ elements for
$\mathbb{GF}$($2^3$) is \textbf{0(by default)}, $\alpha^0 = 1, \alpha^1 = \alpha, \alpha^2 =
\alpha^2, \alpha^3 = \alpha + 1,$ $\alpha^4 = \alpha^2 + \alpha, \alpha^5 =
\alpha^2 + \alpha + 1, \alpha^6 = \alpha^2 + 1$.

A projective space of dimension \textit{d} is denoted by $\mathbb{P}$(d,$\mathbb{F}$)
and consists of one-dimensional subspaces of the (d + 1)-dimensional
vector space over $\mathbb{F}$. Zero-dimensional subspaces of the
projective space are called \textit{points}. The total number of points
in $\mathbb{P}$(d,$\mathbb{F}$) are $P(d)\;=\;\frac{s^{d+1}-1}{s-1}$.
Let us denote the collection of all the \textit{l}-dimensional projective
subspaces by $\Omega_{l}$. To count the number of elements in each
of these sets, we define the function\,\cite{karm1}

\begin{equation}
\phi(n,l,s)\;=\;\frac{(s^{n+1}-1)(s^{n}-1)\cdots(s^{n-l+1}-1)}{(s-1)(s^{2}-1)\cdots(s^{l+1}-1)}
\label{phi_eq}
\end{equation}

The number of m-dimensional subspaces of $\mathbb{P}$(d,$\mathbb{F}$)
is $\phi$(d,m,s). Hence, the number of \textit{l}-dimensional subspaces
contained in an \textit{m}-dimensional subspace(where 0 $\leq$ l
$<$ m $\leq$ d) is $\phi$(m, l, s), while the number of \textit{m}-dimensional
subspaces containing a particular \textit{l}-dimensional subspace
is $\phi$(d-l-1, m-l-1, s). For more details on projective space
construction, refer\,\cite{shulin}.

For example, to generate \textit{Projective Geometry} corresponding to
above Galois Field example ($\mathbb{GF}$($2^3$)), the 2-d
projective plane, we treat each of the above \textit{non-zero} element as
\uline{points} of the geometry. Further, we pick various subfields(vector
subspaces) of $\mathbb{GF}$($2^3$), and label them as various
\uline{lines}. Thus, the 7 lines of the projective plane are \{1, $\alpha$,
$\alpha^3$ = $1+\alpha$\}, \{1, $\alpha^2$, $\alpha^6$ = $1+\alpha^2$\},
\{$\alpha$, $\alpha^2$, $\alpha^4$ = $\alpha^2+\alpha$\}, \{1,$\alpha^4$ =
$\alpha^2+\alpha$, $\alpha^5$ = $\alpha^2+\alpha+1$\},
\{$\alpha$, $\alpha^5$ = $\alpha^2+\alpha+1$, $\alpha^6$ = $\alpha^2+1$\},
\{$\alpha^2$, $\alpha^3$ = $\alpha+1$, $\alpha^5$ = $\alpha^2+\alpha+1$\}
and \{$\alpha^3$ = $1+\alpha$, $\alpha^4$ = $\alpha+\alpha^2$, $\alpha^6$ =
$1+\alpha^2$\}. The corresponding geometry can be seen as figures
\ref{pg_deriv}.


\begin{figure*}[h]
\centerline{\subfloat[Line-point
        Association]{\includegraphics[scale=0.2]{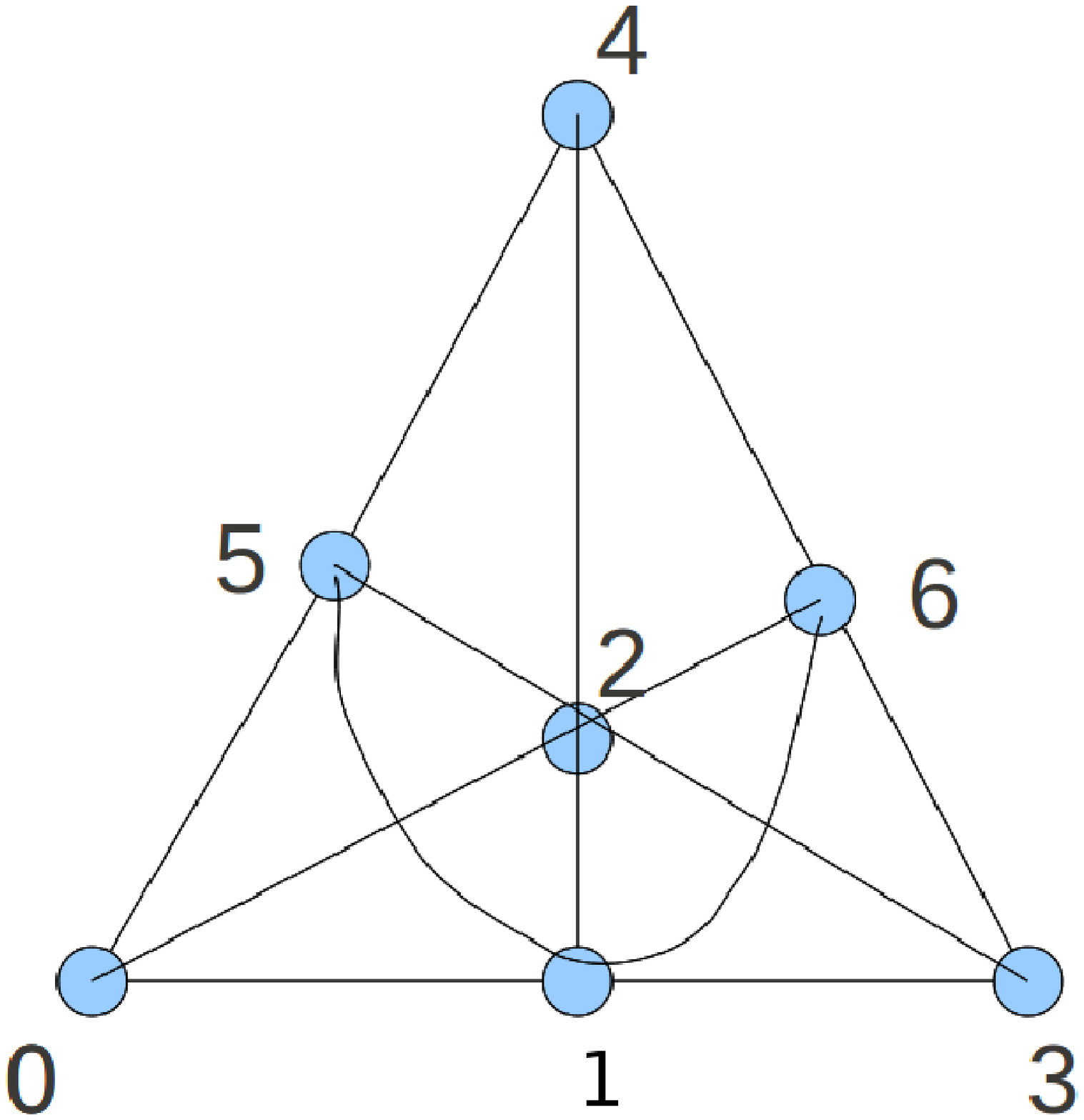}
\label{subfig1}}
\hfil
\subfloat[Bipartite
Representation]{\includegraphics[scale=0.4]{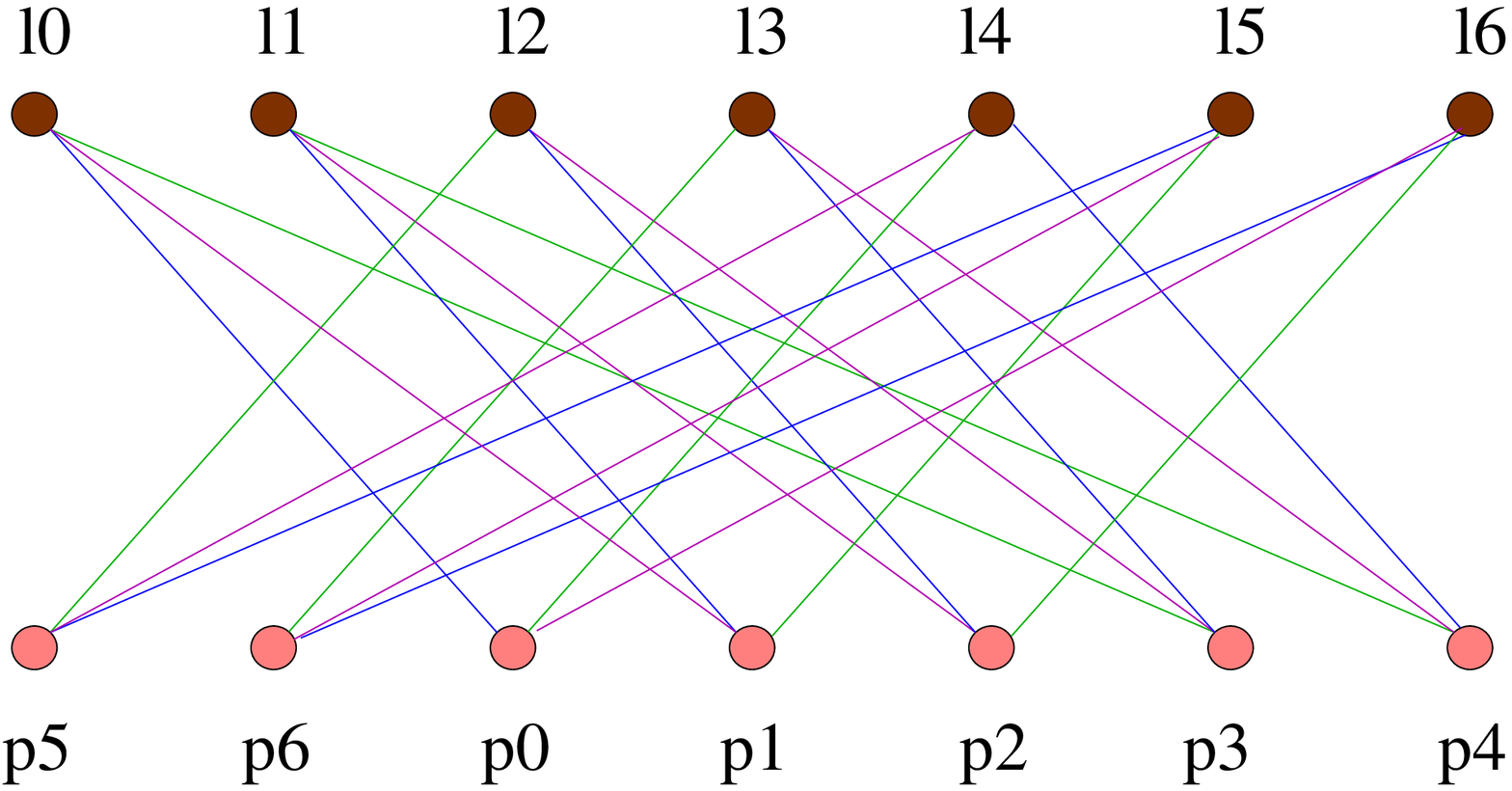}
\label{subfig2}}}
\caption{2-dimensional Projective Geometry and Its Bipartite Representation}
\label{pg_deriv}
\end{figure*}

For Projective spaces, Karmarkar had in past evolved an architecture
for parallel computing. There are problems, such as LU/Cholesky decomposition,
which have been found to be amenable for parallel computation using
processing units and memories, connected using subgraph of an instance
of projective geometry. For such problems, once the \textit{appropriate}
geometry has been identified, a pair of dimensions $d_{m}$ and $d_{p}$
is chosen. The processing units are then bijectively associated with
subspaces of dimension $d_{p}$, while memories are associated with
subspaces of dimension $d_{m}$. A connection between a processing
unit and memory is established if the corresponding subspaces have
a non-trivial intersection\,\cite{karm1}.

In this architecture, accesses to memories happen in a structured
fashion. By exploiting symmetry of the geometry, it is possible to
\textit{pair up all processing units and memories}, such that they
communicate in a conflict-free manner. Each such set of pairs forms
a \textit{perfect-access pattern}. A collection of all such patterns
together forms a \textit{perfect-access sequence}, which ensures that
every processing unit gets to communicate with every memory it is
directly connected to, in an overall computation cycle. Once the problem
is broken down to parallelizable atomic computations, and corresponding
memory blocks for storing data, these computations can then be assigned
to processor \textit{connected to} these relevant memories, which
depends on the problem and the underlying geometry. Load balancing
is ensured by symmetry of the geometry. Thus, data required for computation
is brought in parallely, computations on each processor are carried
out parallely, leading to efficient and conflict-free use of resources.

\subsection{Projective Spaces and Perfect Difference Networks}

Recently, there has been discovery of a new class of parallel interconnect
networks, perfect difference networks\,\cite{rakov1},\,\cite{rakov2}.
The diameter of these graphs being 2, any node is \textit{reachable}
from any other node in one or two hops. For various schemes of routing,
it is argued that this lesser diameter does help in improving the
worst-case communication latency. Borrowing from this basic idea of
reducing latency, the very thought of using a row of nodes in bipartite
network graph as switches or buses to have graph diameter as 2 is
worth looking at. In projective geometry, not every two lines meet
on a plane. But when they meet, their distance is 2. Hence a bus can
be used to communicate between them, something that we use later
in an LU decomposition scheme. A projective plane being a PDN,
the diameter-2 access has been \textit{implicitly} used in the SpMV solving
as well.

Next, we report our \textbf{concrete} investigations in the
application of \textbf{projective geometry} for a prototypical computing
core, viz. preconditioned conjugate gradient(\textbf{PCG}) method for
solving large sparse linear systems.

\section{Advantages of Projective Distribution in PCG Computation}

\label{cg_anal-1}
The preconditioned conjugate gradient method has SpMV as its main computing
subroutine (about $90\%$ of the computing time is spent in SpMV).  PCG
being an iterative solver, SpMV is invoked repeatedly during one invocation
of the solver. Hence, the solver is designed to prefer larger data
movements during the setup, rather than during the iterations.

In a typical SpMV implementation within \textbf{PCG} algorithm, matrix and
vector blocks are distributed to various processors.  Matrix blocks are not
modified during the iterations of PCG, and moreover are larger than vector
blocks. Hence a parallel implementation of the algorithm is structured so
as to move vector- but not matrix- blocks among the processors during the
iterations.  Each processor can perform certain block - vector
multiplications without needing additional data, whereas some
multiplications can be carried out only after receiving an appropriate
vector block.  The \textit{amount of communication} required for these
vector blocks is an \textit{important factor} influencing the performance
of parallel SpMV.

The data distribution that exploits the conflict-free communications of a
PG interconnect topology (dimension 2) identifies memories as points and
computing nodes (processors) as lines in the projective space. A memory
module and a processor are connected if the point corresponding to the
memory module lies on the line corresponding to the processor. The matrix
is distributed so that the block $A_{i,j}$ is allocated to the unique line
passing through points $i$ and $j$. Input sub-vector $x_{i}$ is stored on
processor $i$ and output sub-vector $y_{j}$ on processor $j$. With such a
data distribution, the communication of input and output sub-vectors can be
seen to be conflict free by virtue of presence of perfect access patterns
and sequences. In our experiment, we use the same data distribution, with
the difference that we \textit{do not} focus on the conflict-free nature of
communication\footnotemark[1]\footnotetext[1]{Hence we do not demand a PG-connected network },
instead focusing on the communication load dictated by the data
distribution.

In the \textit{classical} row-wise distribution, every processor needs
$n-1$ vector blocks from $n-1$ processors for it to complete one
multiplication, and hence the communication complexity of this distribution
is $O(n)$ per SpMV. In projective distribution, every processor needs to
communicate $2p\simeq2\sqrt{n}$ vector blocks to complete the
multiplication. Hence, the communication complexity of one SpMV under the
projective distribution is $O(\sqrt{n})$.

It is therefore expected that the performance of projective distribution
should have a \uline{superior scaling behavior} as compared to row-wise
distribution. In next few sections, we provide details of our
\textit{thorough} investigation of this advantage.

Though we have used \textit{preconditioned} CG, the performance of the
solver is dominated by SpMV, and hence the focus of the discussion will be
SpMV rather than the preconditioner.

\section{Design Approach for Novel PCG Solver}

\subsection{Preliminaries}

Within the context of a parallel implementation, the matrix $A$ is viewed
as a \textit{blocked} matrix, and the vectors are viewed as blocked
vectors. Notationally, we represent scalar elements of $A$ as $a_{i,j}$,
and the blocks of the matrix as $A_{k,l}$. This notational convenience of
using lowercase letters for scalars and uppercase letters for matrix blocks
can not be carried over to the vectors, since the convention is to use
lowercase letters for full vector itself. Since the blocking of the various
vectors also becomes relevant within the discussion in this report, we do
need a notation for vector blocks. We introduce the
notation $X_{k}$ to denote a block of the vector
$x$. The set of blocks $A_{i,*}$will be called the $i^{th}$ row-block of
matrix $A$, and similarly the set of blocks $A_{*,j}$will be called the
$j^{th}$column-block of matrix $A$.  A subset of a row-block will be called
a row sub-block and similarly a subset of a column-block will be called a
column sub-block.

The parallel SpMV algorithm distributes the matrix tiles and input-/output-
vector blocks to the processes. If a tile $A_{i,j}$ (or vector block $X_l$)
is allocated to process $k$, we say process $k$ \textbf{owns} the tile (or
the vector block), and is the \textbf{owner process} for the tile (or
block).  The algorithm exploits the associativity in the computation $Y_i =
\sum_{j}{A_{i,j}}\cdot X_j$, to parallely compute multiplications of
individual tiles $A_{i,j}$ with vector-blocks $X_j$.  The results of such
multiplications, which are partial-sums for $Y_i$, can then be added up to
obtain the complete $Y_i$ by its owner process, which defines to be
the same process that owns $X_i$.

Our parallel implementation uses the \textbf{MPI} library for communication
between various processes. We therefore use standard MPI terminology,
\textbf{process} to denote an MPI process usually running on a separate
computational processor, \textbf{rank} to denote the uniquely identifying
integer rank associated with each process, \textbf{rank 0} or \textbf{root}
to denote a rank that is often used in special ways.  Since the PCG study
and benchmarks were primarily focused on the data distribution rather than
the interconnect topology, we shall mostly use the term process instead of
the term ``processor''.

\subsection{Overview of PCG Algorithm}

We reproduce an overview of the reference definition of a PCG, as defined
in \cite{SHEWCHUK94A}, for the sake of easy reference in Algorithm
\ref{alg:pcg}.

\begin{algorithm}
\noindent \textbf{Input}:~Matrix $A$,~vector $b$,~initial vector
$x_{0}$, ~Pre-conditioner $M$,\\
 \hspace*{28pt}Max. number of iterations $i_{max}$, ~Error tolerance
$\epsilon$~$<$~1\\
 \textbf{Output}: ~Vector $x$ satisfying $Ax=b$ in less than~$i_{max}$
iterations with residual dropping\\
 \hspace*{28pt}~by a factor of $\epsilon$\\
{Algorithm}\\
\begin{tabular}{cll}
1 & $i\gets 0$;~$x\gets x_{\text{0}}$; & \tabularnewline
2 & $r\gets b-Ax$; & //~\textbf{SpMV} \tabularnewline
3 & $d\gets M^{-1}r$; & \tabularnewline
4 & $\delta_{new} \gets r^{T}d$; & ~//~dot~product\tabularnewline
5 & $\delta_{0}\gets\delta_{new}$; & \tabularnewline
6 & $\delta_{limit}\gets\epsilon^{2}\delta_{0}$; & // limiting residual\tabularnewline
7 & \multicolumn{2}{l}{\textbf{while}~(~$i<i_{max}$~and~$\delta_{new}>\delta_{limit}$)}\tabularnewline
8 & \hspace{24pt}$q\gets Ad$; & //~\textbf{SpMV}\tabularnewline
9 & \hspace{24pt}$\alpha\gets\delta_{new}/(d^{T}q)$; & \tabularnewline
10 & \hspace{24pt}$x\gets x+\alpha d$; & //~saxpy\tabularnewline
11 & \hspace{24pt}\textbf{if}~(~$i\equiv0$~$mod$~$50$) & \tabularnewline
12 & \hspace{24pt}\hspace{24pt}$r\gets b-Ax$; & //~\textbf{SpMV}\tabularnewline
13 & \hspace{24pt}\textbf{else} & \tabularnewline
14 & \hspace{24pt}\hspace{24pt}$r\gets r-\alpha q$; & //~saxpy\tabularnewline
15 & \hspace{24pt}\textbf{endif}~ & \tabularnewline
16 & \hspace{24pt}$s\gets M^{-1}r$; & \tabularnewline
17 & \hspace{24pt}$\delta_{old}\gets\delta_{new}$; & \tabularnewline
18 & \hspace{24pt}$\delta_{new}\gets r^{T}s$; & //~dot~product\tabularnewline
19 & \hspace{24pt}$\beta\gets\delta_{new}/\delta_{old}$; & \tabularnewline
20 & \hspace{24pt}$d\gets s+\beta d$; & //~saxpy~\tabularnewline
21 & \hspace{24pt}$i\gets i+1$; & \tabularnewline
22 & \textbf{endwhile} & \tabularnewline
\end{tabular}
\caption{PCG Algorithm}
\label{alg:pcg}
\end{algorithm}

To summarize, the algorithm performs an initial SpMV to identify the
initial and hence target residues. Subsequently, each iteration performs
various vector operations and one SpMV operation. Since SpMV performance
dominates the performance of the solver, our discussion will be mostly
about SpMV i.e. steps 2, 8 and 12 of the algorithm.

Our experiment explores various decisions at both the algorithm
and implementation level. To be able to easily compare and isolate
the effect of these decisions, each of these decisions appears as
either compile-time or most often a run-time decision. In such cases,
the choice of one or the other option is indicated at the run-time,
and the performance measurements are recorded for the active set of
parameters. In the remaining part of this section, we will introduce
these decisions. While our experimental performance results reflect
the influence of many of these decisions, in this paper, we analyze
the influence of only the data distribution.

\subsection{Data Distribution}
Conventional matrix data distributions are the 1-D or 2-D block
distributions(\cite{vipin}).  These use either 1 or 2 dimensions\footnotemark[2]\footnotetext[2]{in a Cartesian sense} of the matrix to identify the distribution.
Thus, in block 1-D distribution, $n$ row- or column-blocks are allocated to
each of the $n$ processes, and in block 2-D distribution, a square block
$(i, j), 1 \leq i, j \leq n$ is allocated to each of the $n^2$ processes.
Conventional 1-D and 2-D block distributions can be considered as
``Cartesian'' distributions, since these are based on ``Cartesian''-like
coordinates.

The projective distribution uses $n\times n$ blocks, but a different
relationship between the indices $i$ and $j$ to allocate matrix block $(i,
j)$. We choose row-wise block 1-D distribution\footnotemark[3]\footnotetext[3]{Abbreviated as
row-wise distribution hereafter} for our performance comparison, since its
compute load distribution is very close to the projective distribution.  To
highlight the similarities as well as differences between these two
distributions, we introduce a new classification of data distributions.

In this classification of distributions, the blocking (\textit{or
tiling}) of the matrix remains the same, viz, $n\times n$ square blocks,
and each processor gets $n$ tiles only.  However, we \textit{relax the
constraint} of choosing the tiles \textit{along a \underline{Cartesian}
dimension} only, allowing the tiles to be chosen more generally.  Due to
this relaxation, we call this family of distributions \textbf{Weak
\underline{Cartesian}} distributions. Both the data distributions we
consider -- the row-wise and projective -- satisfy the following: (Note:
the solver is run on $n$ processes)
\begin{definition}[{\bf Weak Cartesian} Conditions]
\begin{enumerate}
\item The $r\times r$ square matrix $A$ is viewed as a blocked matrix,
with $n\times n$ square blocks, each with $\lfloor\frac{r}{n}\rfloor$
rows and columns. The right and bottom boundary blocks $A_{*,n}$ and
$A_{n,*}$ account for all the extra rows and columns in the case
when $r$ is not an integral multiple of $n$.
\item The diagonal block $A_{i,i}$ is allocated to process $i$.
\item Each process is allocated $n-1$ other matrix blocks, apart from its
designated diagonal block.
\item Each block is allocated to a unique process.
\end{enumerate}
\label{defn:weak_cartesian_condition}
\end{definition}
Note that this classification does not require the additional $n-1$ blocks
to be from the same row or the column.  These blocks can be located
anywhere in the matrix, subject to the uniqueness condition. The uniqueness
condition also implies that two diagonal blocks can not be allocated to one
process.

We term data distributions which use $n^2$ square blocks for distribution
among $n$ processes, satisfying the weak cartesian conditions above as an
$n$-process ``\textbf{Weak Cartesian} distribution''.

During the parallel algorithm, a process will perform the tile
multiplication $A_{i,j}\cdot X_j$ only for the  tiles $A_{i,j}$ it owns.
Since the indices $i$ and $j$ are not constrained in any way in a weak
cartesian distribution, a process $k$ may own tiles $A_{i,j}$ where $i \ne
k$ and/or $j \ne k$. When $j \ne k$ (i.e. tiles from other columns), $k$
will need to \textit{receive} the corresponding input
vector block ($X_j$ in this case), before the beginning of computation of partial sums
(\textit{remember  that the matrix tiles are not moved during
multiplication}).  On the other hand, when $i \ne k$ (i.e. tiles from other
rows), $A_{i,j}\cdot X_j$ will be only a partial sum for $Y_i$.  Such a
partial sum will have to be \textit{sent} back to the process that owns
$Y_i$ for completing the computation.

Hence, an SpMV based on a weak cartesian distribution will result in each
process doing some communication to receive the input blocks it needs and
send its input blocks to other processes at the beginning of the
multiplication. After the individual multiplications have been computed,
each process will perform some communication to send partial sums it has
produced for others, and receive partial sums it needs. Different
distributions may require one or both the types of communications.

\subsubsection{Row-wise Distribution}

The row-wise distribution allocates the matrix blocks $A_{i,*}$
to the process $i$. By our convention, process $i$ starts an iteration
with its own vector block $X_{i}$. Subsequently, it needs all the
other vector blocks at some point or other within each iteration to
complete the computation of $Y_{i}$. The low diameter interconnect
topology based on projective spaces, described earlier, can be used
with even the row-wise distribution, but that may not maximize the
communication throughput. Hence, we alternatively use the projective space
structure to define a novel data distribution, described in the next
paragraph. Note the row-wise distribution requires communication
for the input vector blocks only. Since each process has all the tiles
in one row, each process ends up with complete output vector blocks.

\subsubsection{Projective Distribution - A Family of Distributions}

For PCG computation, memory modules as points and processes as lines in a
\textbf{two-dimensional} projective space $\mathcal{P}^{2}$, are used to
define an interconnect scheme between $n=p^{2}+p+1$ processes and $n$
memories for a prime power $p$.  Lines and points are duals of each other
in the projective plane.  Each line has $p+1$ points, and each point is on
$p+1$ lines. See section \ref{pg_details} for an example derivation for
such geometry.

The perfect access patterns and sequences used in communication are defined
using a mapping $c:N\rightarrow N^{p+1}$, where N is the set \{0, 1,
\ldots{}, n-1\}.  The value $c(k)$ corresponds to line $k$, and is the set
of points on this line.
\begin{lemma} 
\label{lempgincidence}
In a \textbf{two-dimensional} projective space $\mathcal{P}^{2}$, for any
$i,j(i\neq j)$,
\begin{itemize}
\item There is a unique line $k=L_{i,j}$  such that $c(k)$ contains both
the points $i$ and $j$ (\emph{Unique line through any two distinct points})
\item There is a unique point $k'=P_{i,j}$  such that the lines $c(i)$ and
$c(j)$ both contain $k'$ (\emph{Unique point common to any two distinct
lines})
\end{itemize}
\end{lemma}

For a given $p$ and hence $n$, this mapping defines a data distribution --
the projective data distribution referenced in earlier sections.  More
generally, by varying $p$, we get a family of 2-D projective distributions,
all of which follow the same data distribution procedure. We now define the
data assignment for a given $p$ (and hence $n$).
~\\
~\\
~\\
\hrule
\begin{algorithmic}[0]
\Procedure{Projective Data Assignment}{$A$}
\State 1. Allocate the diagonal matrix block $A_{i,i}$ to process $i$.
\State 2. To allocate blocks on process $k$, consider the line $k$
corresponding to the process.
\State 3. For every distinct pair of points $i$, $j$ on line $c(k)$,
allocate matrix block $A_{i,j}$
\State (and $A_{j,i}$) to process $k$.
\EndProcedure
\end{algorithmic}
\hrule
~\\
~\\

\subsubsection{Characteristics of Projective Distribution}
We will now study some formal characteristics of the projective
data distribution. Note that as yet, we have not used the fact that the
matrix is sparse in any way, and hence, the results would hold true for
dense matrices as well.
\begin{lemma} 
\label{lempgpart}
The distribution of matrix blocks by procedure Projective Data Assignment
results in a symmetric partition of the matrix blocks.
\end{lemma}
\begin{proof}
\begin{enumerate}
\item Each diagonal block $A_{i,i}$ is allocated to process $i$.
\item Each non-diagonal block $A_{i,j}$ ($i \ne j$) is allocated to a
unique process $k=L_{i,j}$, where $c(k)$ is the unique line through points
i and j (by Lemma \ref{lempgincidence}). Thus, all the blocks are covered,
resulting in a \textit{partition}.
\item Further $A_{j,i}$ is also allocated on the same process $k$, yielding
a \textit{symmetric} partition.
\end{enumerate}
\end{proof}

\begin{lemma} 
\label{lempgpblocks}
The partition of matrix blocks by procedure Projective Data Assignment
results in each process being allocated $O(p)$ row (column-) sub-blocks.
\end{lemma}
\begin{proof}
Since the partition is symmetric, we prove the result for row sub-blocks
and the corresponding result for column sub-blocks will follow from
symmetry.
\begin{enumerate}
\item Consider the non-diagonal blocks allocated to process $k$, and a
        point $i$ incident on line $k$: Each point $l (\neq i)$ on the line
        $c(k)$ will cause a matrix block $A_{i,l}$ to be allocated to
        process $k$. The set of such blocks $\{A_{i,l}: l \in c(k), l \neq
        i\}$ is a row sub-block of A.
\item Thus, each point on the line $k$ results in one row sub-block of
        matrix A. Hence, the $p+1$ points on line $k$ will result in $p+1$
        row sub-blocks.
\item If point $k$ is not incident on line $k$, the diagonal block
        $A_{k,k}$ which is always allocated to process $k$, will introduce
        an additional row sub-block.
\end{enumerate}
\end{proof}

\begin{remark}: 
If a process $k$ is allocated a sub-block $A_{i,j}$, it will need the
vector block $X_j$ to perform its allocated computations, and produce
partial block $Y_i$.
\end{remark}
\begin{remark}: 
>From Lemma \ref{lempgpblocks}, and the remark above, it is obvious, that
each process will require $O(p)$ input blocks, and will produce partial
sums for $O(p)$ output blocks.
\end{remark}

To satisfy this input/output requirement, $O(p)$ units of
communication need to be recieved, and $O(p)$ to be sent. Adding them up,
we get the following \textbf{most important result} of our work.

\begin{lemma} 
Projective Distribution has a communication complexity of $O(p) =
O(\sqrt{n})$ per process.
\end{lemma}

The complexity of $O(p)$ is exact for dense matrix-vector multiplication,
but only an upper bound for SpMV, since some vector-blocks may not need
to be communicated based on the matrix structure in SpMV.

\begin{remark}: 
Due to above distribution, process $k$ contains matrix blocks from
row-blocks $i$ and $j$, and also from column-blocks $i$ and $j$,
where points $i, j \in c(k)$.
\end{remark}

\begin{remark}: 
Due to matrix blocks $A_{i,j}$ and $A_{j,i}$, process $k$ will need vector
blocks $X_{i}$ and $X_{j}$ from their owner processes, and will provide
$X_{k}$ to other processes. Similarly, it will produce partial results
$Y_i$ and $Y_j$; and some other processes will contribute to $Y_k$.
\end{remark}
\begin{remark} 
A row-wise distribution needs to receive $n-1$ vector blocks and produce
only one vector block as result, to complete the block multiplication.
Since it computes the result vector block completely, there is no
communication for partial results. On the other hand, a column-wise
distribution will need only one input vector block to complete the
computation, but will need to send $n-1$ partial results. In both the
cases, the total communication complexity is $n-1=O(n)$ blocks per
process.

As the number of row sub-blocks increases, the communication complexity
for partial results will increase, while the communication complexity
for input vector blocks will reduce. On the other hand, if the number of
column sub-blocks increases, the reverse situation will manifest. The
total communication complexity, which is the sum of number of row
sub-blocks and number of column sub-blocks, will be at a minimum when both
these numbers are roughly equal, which is approximately $\sqrt{n}$.
\end{remark}

\begin{lemma} 
In an $n$-process weak cartesian distribution, if a processor has $r$ row
sub-blocks and $c$ column sub-blocks, then $r\times c \geq n$.
\begin{proof}
\begin{enumerate}
\item Consider process $k$ and the minimal rectangular sub-matrix of A,
which has
\begin{itemize}
\item all the blocks allocated to $k$,
\item no entirely empty rows and no entirely empty columns.
\end{itemize}
\item This rectangular submatrix has $r$ rows and $c$ columns, and hence
$r\times c$ blocks of the complete matrix. Note some of these blocks may
not be allocated to process $k$.
\item Since all the $n$ blocks allocated to  process $k$ are in the
submatrix, $r\times c \geq n$.
\end{enumerate}
\end{proof}
\end{lemma}

The construction of this submatrix is illustrated in
Table~\ref{table:minsubmat}. Rows 1, 4 and 9 are  completely empty, and so
are columns 1, 5, 8, 9. These entirely empty rows and columns are removed
to yield the minimal submatrix described above.  Notice in this example,
that $n=9$, $r=6$ and $c=5$ and indeed $5\times 6 = 30 > 9$.

\begin{table}[b]
\small
\begin{tabular}{|c|c|c|c|c|c|c|c|c|c|}
\hline
 & 1 & 2 & 3 & 4 & 5 & 6 & 7 & 8 & 9\tabularnewline
\hline
\hline
1 &  &  &  &  &  &  &  &  & \tabularnewline
\hline
2 &  &  & x &  &  & x &  &  & \tabularnewline
\hline
3 &  &  &  & x &  &  & x &  & \tabularnewline
\hline
4 &  &  &  &  &  &  &  &  & \tabularnewline
\hline
5 &  &  &  & x &  & x &  &  & \tabularnewline
\hline
6 &  &  &  &  &  &  & x &  & \tabularnewline
\hline
7 &  & x &  &  &  &  &  &  & \tabularnewline
\hline
8 &  &  &  &  &  & x &  &  & \tabularnewline
\hline
9 &  &  &  &  &  &  &  &  & \tabularnewline
\hline
\end{tabular}
\hspace*{72pt}
\begin{tabular}{|c|c|c|c|c|c|}
\hline
 & 2 & 3 & 4 & 6 & 7\tabularnewline
\hline
\hline
2 &  & x &  & x & \tabularnewline
\hline
3 &  &  & x &  & x\tabularnewline
\hline
5 &  &  & x & x & \tabularnewline
\hline
6 &  &  &  &  & x\tabularnewline
\hline
7 & x &  &  &  & \tabularnewline
\hline
8 &  &  &  & x & \tabularnewline
\hline
\end{tabular}

\caption{Sample allocation -- $9\times 9$ matrix, $6\times 5$ minimal
submatrix}
\label{table:minsubmat}
\end{table}

Note that in our discussion of Projective distribution, we have not used
``sparsity'' of the matrix. We state and prove the next two results too,
for ``general'' matrices (potentially dense). We thus speak of ``MV''
or Matrix Vector multiplication, instead of particularly SpMV.

\begin{lemma} 
Matrix-Vector multiplication on any weak cartesian distribution has a
per-process communication complexity bounded below by $O(\sqrt{n})$, where
$n$ is the number of processes.

\begin{proof}
\begin{enumerate}
\item Consider process $i$, which has $n$ tiles allocated to it.
\item Let $r$ be the number of row sub-blocks and $c$ be the number of
column sub-blocks on process $i$. Then the communication complexity at
process $i$ will be $r+c$ ($c$ communications to receive input sub-vectors
and $r$ communications to send output sub-vector partial sums).
\item Since $r\times c \geq n$, $c \geq \frac{n}{r}$
\item Hence, the total communication complexity is bounded below by
$r+\frac{n}{r}$, as r varies.
\item $r+\frac{n}{r}$ will be minimum ($=2\sqrt{n}$) when $r=\sqrt{n}$.
\end{enumerate}
\end{proof}
\end{lemma}

Again, based on matrix structure in SpMV, some vector-block communication
will be elided, the result may not hold for \textit{all} sparse matrices

Since matrix-vector multiplication on projective distribution has
communication complexity of $O(\sqrt{n})$, which is the lower bound for
matrix-vector multiplication on weak cartesian distributions, we have the
next theorem.
\begin{theorem} 
For matrix-vector multiplication, projective distribution has the minimum
communication complexity among the weak cartesian distributions.
\end{theorem}

Based on the reasons mentioned earlier, this result too may not hold in
its entirety for sparse matrices. Characterizing the sparse matrices
with respect to this result is work in progress.

Observe that a process $k$ needs vector block $X_{j}$ only if the process
is allocated a block $A_{i,j}$ for some $i$ at some point of time. Further,
such a block can be allocated to process $k$ if and only if the point $j$
is on line $k$. Since every line has only $p+1$ points, obviously, a
process $k$ will need only $p+1$ (or $p$ if point $k$ is also on line $k$)
vector blocks.

By similar logic, a process $k$ will produce partial sums for only $p+1$
vector blocks. These characteristics of the projective distribution have
been explained in the following example.

\subsubsection{Illustration of Projective Distribution}

We illustrate this distribution with a connection scheme for p=3 in
table\,\ref{tab:pgdistconn}. It is easy to verify that the incidence
properties (one line through any two points and vice versa) mentioned
above hold for connection scheme. Table \ref{tab:pgdistblocks}
shows the matrix blocks allocated to each process. Table \ref{tab:pgdd}
shows distribution in a matrix form; $(i,j)^{th}$ entry indicates
the process to which the block $A_{i,j}$ is allocated. To
make the pattern visible, blocks allocated to process 0 have been shown
in a box (e.g. \colzero{0}), and those allocated to process 4 been marked
with a bullet (e.g. \colone{4}) in Table \ref{tab:pgdd}. From these tables,
it is easy to verify:
\begin{enumerate}
\item Incidence properties i.e. existence of one line through any two
points, and a point on any two lines. This can be seen by observing in
Table \ref{tab:pgdistconn} that
\begin{table}[h]
\caption{PG Connections}
\label{tab:pgdistconn}
\centering
\begin{tabular}{|r|rrrr|} \hline
Line  & \multicolumn{4}{c|}{Points}\tabularnewline \hline
0 & 6& 7& 9& 2\tabularnewline
1 & 7& 8& 10& 3\tabularnewline
2 & 8& 9& 11& 4\tabularnewline
3 & 9& 10& 12& 5\tabularnewline
4 & 10& 11& 0& 6\tabularnewline
5 & 11& 12& 1& 7\tabularnewline
6 & 12& 0& 2& 8\tabularnewline
7 & 0& 1& 3& 9\tabularnewline
8 & 1& 2& 4& 10\tabularnewline
9 & 2& 3& 5& 11\tabularnewline
10 & 3& 4& 6& 12\tabularnewline
11 & 4& 5& 7& 0\tabularnewline
12 & 5& 6& 8& 1\tabularnewline
\hline
\end{tabular}
\end{table}
\begin{table}[h]
\caption{Projective Data Distribution - Blocks for each Process}
\label{tab:pgdistblocks}
\centering
{\tiny
\begin{tabular}{|c|ccccccccccccc|} \hline 
Proc. & \multicolumn{13}{c|}{Allocated Blocks}\tabularnewline \hline 
0&0,0&6,7&6,9&6,2&7,6&7,9&7,2&9,6&9,7&9,2&2,6&2,7&2,9\tabularnewline
1&1,1&7,8&7,10&7,3&8,7&8,10&8,3&10,7&10,8&10,3&3,7&3,8&3,10\tabularnewline
2&2,2&8,9&8,11&8,4&9,8&9,11&9,4&11,8&11,9&11,4&4,8&4,9&4,11\tabularnewline
3&3,3&9,10&9,12&9,5&10,9&10,12&10,5&12,9&12,10&12,5&5,9&5,10&5,12\tabularnewline
4&4,4&10,11&10,0&10,6&11,10&11,0&11,6&0,10&0,11&0,6&6,10&6,11&6,0\tabularnewline
5&5,5&11,12&11,1&11,7&12,11&12,1&12,7&1,11&1,12&1,7&7,11&7,12&7,1\tabularnewline
6&6,6&12,0&12,2&12,8&0,12&0,2&0,8&2,12&2,0&2,8&8,12&8,0&8,2\tabularnewline
7&7,7&0,1&0,3&0,9&1,0&1,3&1,9&3,0&3,1&3,9&9,0&9,1&9,3\tabularnewline
8&8,8&1,2&1,4&1,10&2,1&2,4&2,10&4,1&4,2&4,10&10,1&10,2&10,4\tabularnewline
9&9,9&2,3&2,5&2,11&3,2&3,5&3,11&5,2&5,3&5,11&11,2&11,3&11,5\tabularnewline
10&10,10&3,4&3,6&3,12&4,3&4,6&4,12&6,3&6,4&6,12&12,3&12,4&12,6\tabularnewline
11&11,11&4,5&4,7&4,0&5,4&5,7&5,0&7,4&7,5&7,0&0,4&0,5&0,7\tabularnewline
12&12,12&5,6&5,8&5,1&6,5&6,8&6,1&8,5&8,6&8,1&1,5&1,6&1,8\tabularnewline
\hline
\end{tabular}
}
\end{table}
\begin{table}[h]
\caption{Projective Distribution Example - Process owning each Block}
\label{tab:pgdd}
\centering
\begin{tabular}{|r|rrrrrrrrrrrrr|} \hline 
Row & \multicolumn{13}{c|}{Column}\tabularnewline
\cline{2-14}& 0 & 1 & 2 & 3 & 4 & 5 & 6 & 7 & 8 & 9 & 10 & 11 & 12 \tabularnewline\hline
0 &\colzero{0}&7&6&7&11&11&\colone{4}&11&6&7&\colone{4}&\colone{4}&6\tabularnewline
1 &7&1&8&7&8&12&12&5&12&7&8&5&5\tabularnewline
2 &6&8&2&9&8&9&\colzero{0}&\colzero{0}&6&\colzero{0}&8&9&6\tabularnewline
3 &7&7&9&3&10&9&10&1&1&7&1&9&10\tabularnewline
4 &11&8&8&10&\colone{4}&11&10&11&2&2&8&2&10\tabularnewline
5 &11&12&9&9&11&5&12&11&12&3&3&9&3\tabularnewline
6 &\colone{4}&12&\colzero{0}&10&10&12&6&\colzero{0}&12&\colzero{0}&\colone{4}&\colone{4}&10\tabularnewline
7 &11&5&\colzero{0}&1&11&11&\colzero{0}&7&1&\colzero{0}&1&5&5\tabularnewline
8 &6&12&6&1&2&12&12&1&8&2&1&2&6\tabularnewline
9 &7&7&\colzero{0}&7&2&3&\colzero{0}&\colzero{0}&2&9&3&2&3\tabularnewline
10 &\colone{4}&8&8&1&8&3&\colone{4}&1&1&3&10&\colone{4}&3\tabularnewline
11 &\colone{4}&5&9&9&2&9&\colone{4}&5&2&2&\colone{4}&11&5\tabularnewline
12 &6&5&6&10&10&3&10&5&6&3&3&5&12\tabularnewline
\hline
\end{tabular}
\end{table}

\begin{enumerate}
\item Given
any two rows, there is one common entry between the two. e.g. Rows 2 and 6
have the entry ``8'' in common.
\item Given any two entries (numbers), there is one row containing both the
entries.  e.g. Entries 1 and 5 are found in row ``12''.
\end{enumerate}
\item Number of vector blocks required (input) and produced (output) are
p+1. Note the computations performed by process 0 as an example :

\begin{enumerate}
\item Partial sum $\mathbf{y_{0}}\gets A_{0,0}\cdot\mathbf{x_{0}} $,
\item Partial sum $\mathbf{y_{2}}\gets A_{2,6}\cdot\mathbf{x_{6}}+A_{2,7}\cdot\mathbf{x_{7}}+A_{2,9}\cdot\mathbf{x_{9}}$,
\item Partial sum $\mathbf{y_{6}}\gets A_{6,2}\cdot\mathbf{x_{2}}+A_{6,7}\cdot\mathbf{x_{7}}+A_{6,9}\cdot\mathbf{x_{9}}$,
\item Partial sum $\mathbf{y_{7}}\gets A_{7,2}\cdot\mathbf{x_{2}}+A_{7,6}\cdot\mathbf{x_{6}}+A_{7,9}\cdot\mathbf{x_{9}}$,
\item Partial sum $\mathbf{y_{9}}\gets A_{9,2}\cdot\mathbf{x_{2}}+A_{9,6}\cdot\mathbf{x_{6}}+A_{9,7}\cdot\mathbf{x_{7}}$.
\end{enumerate}

Thus, process 0 can perform all its computations using only 4 blocks
$\mathbf{x_{2}}$, $\mathbf{x_{6}}$, $\mathbf{x_{7}}$ and $\mathbf{x_{9}}$,
and produces only 4 blocks $\mathbf{y_{2}}$, $\mathbf{y_{6}}$,
$\mathbf{y_{7}}$ and $\mathbf{y_{9}}$. Since the example under
consideration has p=3, the number of blocks, 4, can be easily seen to be
p+1.
\end{enumerate}
These characteristics are not particular for our choice of the prime
number, and a similar example for \textbf{any} prime (rather, prime power)
can be easily constructed.

\subsubsection{Parallel Algorithm}

\label{ssec:commvecavail-1} Since SpMV is the strongest computational
component within a preconditioned CG (PCG) solver, the parallel PCG
algorithm has as one of its core component the parallel SpMV algorithm.
Note that the matrix blocks need to be distributed to the individual
processes \emph{only} at the beginning of the solver. Subsequent iterations
of the solver do not need the matrix blocks to be re-transmitted.  The
vector x, which keeps on changing for each iteration of the solver, needs
to be sent to all the processes for each iteration.

\vspace{1pt}
~\\
\hrule
\begin{algorithmic}[0]
\Procedure{CG using row-wise distribution}{}
\State \textit{Initial Step:} Let each process $i$ initiate the block
computation $A_{i,i}X_{i}$, as a partial sum
\State for $Y_{i}$.
\While{$i < n$}
\State Let each process participate in sending and receiving different
blocks
\State \hspace*{20pt}of the vector with other processes.
\If{block $X_{k}$ is received}
\State Schedule the computation(s) using $X_{k}$ viz corresponding to local
blocks $A_{*k}$,
\State \hspace*{20pt}thus providing other partial sums.
\EndIf
\State $i \gets i+1$
\EndWhile
\EndProcedure
\end{algorithmic}

\hrule
~\\
\hrule
\begin{algorithmic}[0]
\Procedure{CG using a projective distribution}{}
\State \textit{Initial Step:} Let each process $i$ initiate the block
computation $A_{i,i}X_{i}$, as a partial sum
\State for $Y_{i}$.
\While{$i < (p+1)$}
\State Let each process participate in sending and receiving different
blocks
\State \hspace*{20pt}of the vector with other processes.
\If{block $X_{k}$ is received}
\State Schedule the computation(s) using $X_{k}$ viz corresponding to local
blocks $A_{*k}$,
\State \hspace*{20pt}thus providing other partial sums.
\EndIf
\State $i \gets i+1$
\EndWhile
\Statex
\While{$i < (p+1)$}
\State Send the non-local partial sums produced on $i$ are sent to their
respective owner
\State \hspace*{20pt}processes.
\State At process $i$, add up the received non-local partial sums, produced
by other
\State \hspace*{20pt}processes for $i$.
\State $i \gets i+1$
\EndWhile
\EndProcedure
\end{algorithmic}
\hrule
~\\

It is obvious from the algorithm definitions that the communication
complexity for the row-wise distribution is $\mathbf{O(n)}$, while for the
projective distributions it is $\mathbf{O(\sqrt{n})}$.

\subsection{Vector Communication}

Every SpMV execution requires a communication of vector blocks for the
block computation to complete. The communication of vector blocks can be
done in a variety of  ways. An implementation can choose to broadcast the
entire vector at the beginning of SpMV. However, this alternative suffers
from two drawbacks:
\begin{enumerate}
\item The assumption that each process needs all the vector blocks may not
be valid. As we have seen earlier, this assumption does not hold for
projective distributions.
\item Even otherwise, all the processes wait for entire vector to be
available, before carrying out any block-level multiplication. As the
number of processes increases, the size of a vector block required for one
block multiplication reduces significantly. Hence, broadcasting the vector
introduces a large overhead, which can be avoided. Since the multiplication
of one matrix block needs only one vector-block and not the entire vector,
the parallelism can be increased by scheduling vector communications at a
block level.
\end{enumerate}

Our implementations therefore do not use broadcast alternatives. Instead,
each of the vector blocks required is explicitly communicated. The overhead
of waiting for data to be available is relatively low.

\subsection{Result Vector Ownership}

Apart from the distribution of matrix A and vector $x$, the distribution of
the resultant vector $y$ also influences the communication in PCG. In the
row-wise distribution, since every process owns complete rows of the
matrix, at the end of SpMV, each process ends up with complete elements
(i.e. \textit{not} partial sums) of the resultant vector $y$. Consequently,
the dot-products and saxpy's in the PCG algorithm can be carried out in
parallel on all the processes, using the vector blocks for all the relevant
vectors. This approach still requires the processes to reduce the partial
dot-products, but the communication required in this case is of a single
scalar after every dot product instead of the full vector. As a result,
after SpMV, the vector blocks computed by the processes do not need to be
collected on a single process. This approach can thus save significant
amount of unnecessary communication.

Since the projective distributions allocate blocks from different rows and
different columns, the vector blocks resulting are only partial sums.
Communication steps are therefore necessary for these partial sums to be
summed up together. However, even including this additional overhead, the
communication complexity of the projective distribution is \textbf{far
superior} to the row-wise distribution.

Hence our implementations consistently use the approach of not storing
the entire vector on a single process, but keeping the resultant vector
blocks on their owner processes.

Note that this optimization becomes relevant only when optimizing SpMV in
context of the PCG algorithm, and may not be relevant at all when SpMV is
considered stand-alone.

\subsection{Packing Matrix, Input/Output Vector blocks}

As explained above, the processes communicate vector blocks with other
processes. When a process $q$ sends the vector block $X_{q}$ to
another process $s$, some of the communication may be unnecessary.
Since the matrix is sparse, it is quite likely that the processes
$s$ does not use all the elements of $X_{q}$ when multiplying the
blocks. If $q$ knows which columns of $X_{q}$ are required by $s$
(or by every other process in general), each such communication can
be made more lightweight by sending only the required vector elements
instead of the entire vector block. The information about which matrix
block (and hence its owning process) needs which elements depends
upon matrix structure and hence can be computed during problem
setup. In the simplest form of this approach, the sending process
packs the vector block when sending and the receiving process unpacks
the vector on receipt.

In the case of projective distribution, the same concept applies to the
result vector block (partial sum). The rows that are completely zero will
not contribute to the result vector block, and hence need not be
communicated when reducing the partial sums at the end of SpMV.

This optimization turns out to be more powerful, when it is extended
further to the matrix blocks, particularly the non-diagonal blocks.
When storing a non-diagonal matrix block $A_{i,j}$, the all-zero
rows \textbf{within this block} can be eliminated and the row numbers
within this block renumbered. Similarly, the zero-columns can be eliminated
and column numbers renumbered. When a block is completely packed this
way, it turns out that the vector-unpacking step becomes redundant.
Thus, the sending process still packs the vector block based on the
structure at the receiving process, but the receiving process can
use this packed vector block without unpacking it, since the column
numbers in the packed matrix block and column numbers in the packed
vector block are identical.

A projective
distribution using packed matrix blocks will produce result vector
blocks which are themselves packed. Hence, the process sending the
partial sums incurs no additional packing/unpacking overhead. The
receiving process unpacks the received partial sum based on the structure
information.

\section{Method and Experimental Results for PCG Performance Evaluation}

The parallel CG (with Jacobi preconditioner) implementation was done
in C++ for both row-wise and projective distributions.
Both these implementations were carried out by the same researcher,
using the same common code base for common operations. This ensured
that the optimization style in both the codes was the same, and also
that the benefit of common optimizations was uniformly available to
both the distributions uniformly.

The experiment was carried out on the EKA cluster at Computational
Research Laboratories. The individual nodes are 8-core Intel Clovertown
(Xeon X5365 @ 3GHz), with 16 GB RAM and 4x4MB shared L2 cache, running
HP XC. The performance was measured using different number of processes
to solve a problem. The number of processes chosen were suitable for the
projective
distribution (viz. of the form $p^{2}+p+1$ where p is a prime
power). Thus, the timings have been measured on 7, 13, 21, 31, 57,
73, $\ldots$ blades. During the runs, 4 threads were used on each
core. The sparse matrices in the dataset have been chosen from the
University of Florida matrix market\,\cite{davis}.

The time required for various steps has been measured in terms of
cycles as returned by ``rdtsc'' instruction, which are then converted
to time in seconds based on initial calibration. Operations such as
reading the matrix, preprocessing the matrix, distributing it are
part of the algorithm setup, and hence are not included in the timings.
Some raw simulation details can be found in Table \ref{table: MatSim_Details}.

\begin{table}[h]
\caption{Matrix Simulation Details}

\label{table: MatSim_Details}
\centering
\begin{tabular}{|l|c|c|r@{.}l|r@{.}l|} \hline
\multirow{2}{*}{Matrix}&\multirow{2}{*}{No. Rows}&\multirow{2}{*}{No. Non-zeroes}&
\multicolumn{4}{c|}{Time (sec)}\tabularnewline
\cline{4-7}
&&& \multicolumn{2}{c|}{7 Processes}& \multicolumn{2}{c|}{91 Processes}\tabularnewline \hline
msdoor  & 415,863 & 20,240,935 & 9& 0& 3&0 \tabularnewline \hline
F1  & 343,791 & 26,853,113 & 27 &0& 9& 0\tabularnewline \hline
crankseg\_2 & 63,838 & 14,148,858 & 1&8 & 0&55 \tabularnewline \hline
Benelechi1 & 345,874 & 13,050,496 & 1&0 & 0&4 \tabularnewline \hline
audikw\_1 & 943,695 & 77,651,847 & 30&0 & 6& 0\tabularnewline \hline
af\_shell4 & 504,855 & 19,188,875 & 2&4 & 0&8 \tabularnewline \hline
af\_0\_k101 & 503,625 & 17,550,675 & 1&4 & 0&5 \tabularnewline \hline
\end{tabular}
\end{table}

\begin{figure}[!h]
\begin{centering}
\includegraphics[scale=0.58]{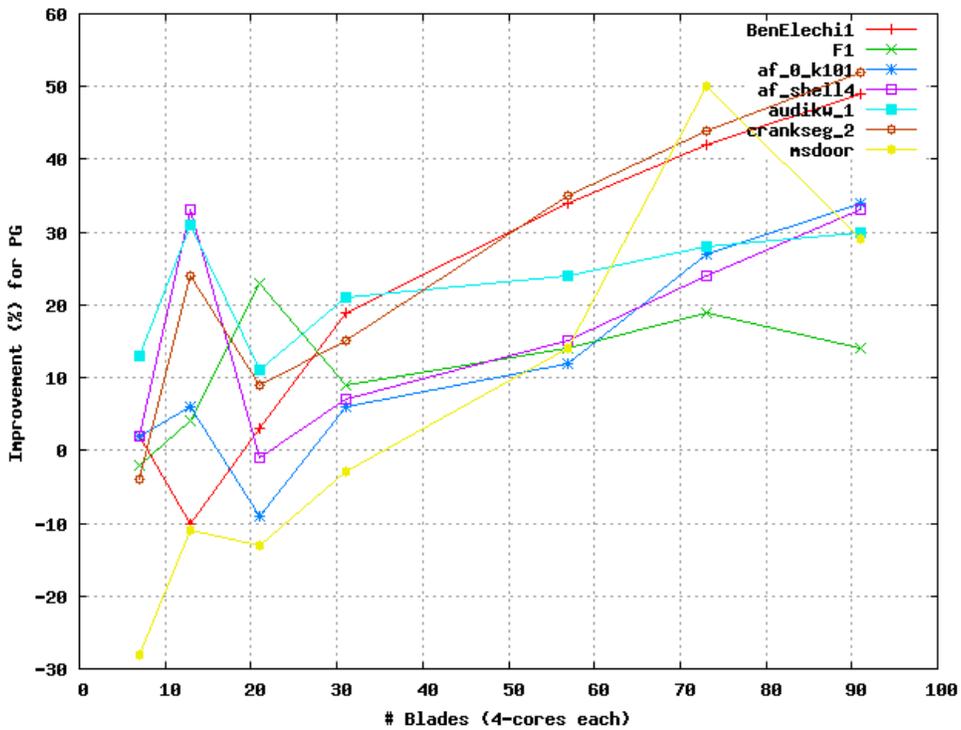}
\par\end{centering}
\caption{Percentage Improvement in projective Distribution w.r.t. row-wise}
\label{somecomp-1}
\end{figure}
\begin{table*}[h]
\caption{Percentage Improvement while using projective geometry based distribution}
\label{table:percimp-1-1}
\centering
\begin{tabular}{|l|r|r|r|r|r|r|r|} \hline
\multirow{2}{*}{Matrix}&\multicolumn{7}{c|}{Percentage Improvement using projective}\tabularnewline
\cline{2-8}&7&13&21&31&57&73&91\tabularnewline\hline
BenElechi1& 2 &-10 &3 &19& 34& 42 &49\tabularnewline\hline
F1 & -2 &4 &23& 9& 14& 19& 14\tabularnewline\hline
af-0-k1-1 & 2& 6& -9& 6 &12& 27& 34\tabularnewline\hline
af-shell4 & 2 &33 &-1 &7 &15& 24& 33\tabularnewline\hline
audikw-1 & 13 &31& 11& 21 &24 &28 &30\tabularnewline\hline
crankseg-2 & -4 &24 &9 &15& 35& 44& 52\tabularnewline\hline
msdoor & -28& -11 &-13 &-3 &14 &50 &29\tabularnewline\hline

\end{tabular}
\end{table*}

The plots show the number of processes on X-axis and the time in seconds
on the Y-axis. Figure \ref{somecomp-1} plots the performance improvement
when using projective geometry distribution against when using the
row-wise distribution. For CG algorithm's computation, performance
improvement rate is higher when the number of processes is lower. As the
number of processes increase, the rate of improvement drops in both the
row-wise and projective geometry distributions.
In row-wise distribution, the performance flattens around 20-30 processes,
and actually starts degrading
beyond 50 processes. In projective geometry distribution, the performance
starts degrading only in a few graphs, and that too beyond 70 processes.
In most of the cases, the flattening goes much beyond 50 processes.

In the next part of the paper, we report our investigations in the
application of \textbf{projective geometry} for another prototypical
computing core, viz. \textbf{LU/Cholesky Decomposition}.

\section{Overview of LU/Cholesky Decomposition}
\label{lu_ov}

We use the popular trailing matrix update algorithm for decomposition
\,\cite{vipin}. The matrices involved in circuit simulation are
generally extremely big in size, hence only a (matrix\nobreakdash-)block
level parallelization can be cost-effective. It also improves the
probability of even distribution of computational load. To generalize
the algorithm to block level, consider the partitioning of the $N\times N$
matrix $A$ into square $B=\frac{N}{b}$ blocks of size $b\times b$.
One can similarly partition $L$ and $U$ matrices into blocks $L_{i,j}$
and $U_{i,j}$ respectively, where $i,j\in{0,1,\ldots,B-1}$. Clearly,
$L_{i,j}$ and $U_{j,i}$, where $j>i$ will be \textbf{0} matrices.
Assuming well-conditioned matrices, the generalized algorithm without
pivoting is then as follows.

\begin{algorithm}
\caption{Block LU Decomposition Algorithm}
\label{lu_algo}
\begin{algorithmic}[1]
\For{$i \leftarrow 0\;\mathbf{to}\; B-1$}
      \State $A_{i,i}$ $\leftarrow$ blockLU($A_{i,i}$)
    \Statex
    \For{$j \leftarrow i+1 \;\mathbf{to}\; B-1$}
      \State $L_{j,i}$ $\leftarrow$ $A_{j,i}U_{i,i}^{-1}$
    \EndFor
    \Statex
    \For{$k \leftarrow i+1 \;\mathbf{to}\; B-1$}
      \State $U_{i,k}$ $\leftarrow$ $L_{i,i}^{-1}A_{i,k}$
    \EndFor
    \Statex
    \For{$j \leftarrow i+1 \;\mathbf{to}\; B-1$}
        \For{$k \leftarrow i+1 \;\mathbf{to}\; B-1$}
          \State $A_{j,k}$ $\leftarrow$ $A_{j,k}$ - $L_{j,i}U_{i,k}$
        \EndFor
    \EndFor
\EndFor

\end{algorithmic}
\end{algorithm}

In this algorithm, the decomposition is \textit{in-place}, hence the
$L_{i,j}$, $U_{i,j}$ matrices can be obtained from decomposed $A_{i,j}$
itself. In each iteration of main loop(line (1)), the operations to
be carried out, and which we have optimized, are:
\begin{itemize}
\item \textbf{{Row/Column Update}} In the $i^{th}$ iteration, the blocks
        of $i^{th}$ row/column are updated(lines (4), (7)). These updates
can be characterized by a triplet \{$(i,j,i)$ or $(i,i,k)$: $i<j,k\leq B-1$\}
of indices.
\item \textbf{{Trailing Matrix Update}} In the $i^{th}$ iteration, blocks
        \{$A_{j,k}$: $i<j,k\leq B-1$\} are updated(line (11)). This is the
most \textbf{dominant} operation of this entire computation. These
updates can be characterized by the triplet $(i,j,k)$(or $(i,j,j)$
when j=k) of indices.
\end{itemize}
While parallelizing this computation, one can assign processing units
which work on different values of indices \textit{j} and \textit{k},
every $i^{th}$ iteration. This way, processing units will do independent
computations simultaneously, every iteration. In such a case, these
processing units will need blocks $L_{j,i}$ and $U_{i,k}$, and these
blocks have one index in common with the block to be updated. This
property is exploited in adapting this algorithm for Karmarkar's architecture,
which we discuss in next few sections.

The \textit{Cholesky decomposition} is another direct matrix decomposition
method, $\mathbf{A}$ = $\mathbf{L}\cdot\mathbf{L^{T}}$, where $\mathbf{L}$
is a lower triangular matrix. The computation of $\mathbf{L}$ follows
nearly identical pattern of the 3 main steps as of LU decomposition
per iteration. The major advantage of Cholesky computation is that
\textit{pivoting} is \textbf{not} required. Hence our schemes are
expected to give clearer, better results, when adapted to Cholesky
decomposition. However to save space, we only refer to LU decomposition
in the forthcoming sections.

\section{Projective Space Details for LU/Cholesky Decomposition}

The topology for scheduling trailing matrix update is governed by
choice of appropriate subspaces $\Omega_{i}$ of some projective space
and their subsumption relations. This relation guides the specification
of interconnection network for the multiprocessing system. Our choice
for topology is a \textit{modified form} of the interconnect proposal
in \,\cite{karm1}. We map the block row and column indices of non-decomposed
matrix $A$ to \textbf{points} of some projective geometry. The distributed
(main) memory blocks as well as processors are mapped to the \textbf{lines},
while either the computation or the communication gets mapped to the
\textbf{planes} of the same geometry. Incidence relationships of lines onto
planes is used to design the connections within the system.
~\\
\hrule
\begin{algorithmic}[0]
\Procedure{Assignment Procedure}{$A$}
\State 1.  Store block $A_{i,j}$ in the memory representing the line
joining points associated
\State with indices \textit{i} and \textit{j}.
\State 2. Assign each computation characterized by a value of indices'
triplet(\textit{i,j,k}) to the
\State processor-memory pair associated with a line
mapped to plane corresponding to
\State the triplet.
\EndProcedure
\end{algorithmic}
\hrule
~\\
\begin{lemma}
A 4-dimensional projective space, ${\mathbb{P}}(4,GF(2))$, has same
number of lines and planes.
\end{lemma}
By above easy-to-prove lemma, our choice of 31-point 4-d projective
geometry, generated over binary field, leads to a symmetric scheme
involving same number of computations as well as processors. Some
of the combinatorial numbers related to ${\mathbb{P}}(4,GF(2))$ can
be found by evaluating function $\phi(.,.,.)$(c.f. section \ref{pg_details}):
~\\
\hrule
\begin{itemize}
\item There are $31$ points, $155$ lines ($\phi(4,1,2)$) and $155$ planes
($\phi(4,2,2)$) in the geometry.
\item Each line has $3$ points on it while a plane has $7$ points.
\item Each plane has $7$ lines, and exactly $3$ lines belonging to a plane
pass through any point.
\item A line is incident on $7$ different planes.
\item A point is present on $15$ different lines ($\phi(3,0,2)$) and $35$
different planes ($\phi(3,1,2)$).
\end{itemize}
\hrule

\subsection{Projective Space Automorphisms}

\textit{Projective subspaces are known to be sets of points}. We
use two automorphisms on these points, namely \textit{Frobenius}
and \textit{Shift} automorphisms, to derive schedules.

The function corresponding to \textbf{Frobenius} automorphism is
$\Phi(x)=x^{p}$, where $x\in\mathbb{F}$
and $p$ is the characteristic of $\mathbb{F}$. Application of this
automorphism in $\mathbb{P}(4,GF(2))$ corresponds to doubling of
index of each point modulo 31 and taking its remainder modulo 31.
Repeated application leads to $5$ different automorphisms using the
Frobenius map.

Similarly, for \textbf{Shift} automorphism, a `shift' function
can be defined on points as $L_{x}:(0,x^{i})\rightarrow(0,x^{i+1}),\forall i\in0,1,\ldots,30$,
where $x$ is the generator polynomial. Clearly, application of $L_{x}$
corresponds to incrementing the index of a point by 1, modulo 31.
As earlier, repeated application of $L_{x}$ leads to 31 different
automorphisms. The most important advantage of working with these
two automorphisms is that in a 4-d projective space ${\mathbb{P}}(4,GF(2))$,
starting with a particular line or plane, it is \textbf{possible}
to enumerate \textbf{all other} lines and planes using these two.
These concepts will be later applied in establishing the interconnection
network.

\subsection{Perfect Matching Patterns}

Using the automorphisms described above, we develop a sequence of
patterns depicting perfect matchings in the bipartite graph made of
lines and planes in ${\mathbb{P}}(4,GF(2))$. We denote a k-d projective
subspace by a k-tuple of points. Since each plane subsumes $7$ different
lines, there can be a \textit{sequence} of $7$ patterns based on
$7$ different matchings. Define the first pattern as $S_{1}$ : $\Omega_{2}$
$\rightarrow$ $\Omega_{1}$, such that
\begin{equation}
S_{1}(p)=L_{x}^{a}(\Phi^{b}(0,1,18)),\textit{ if }p=L_{x}^{a}(\Phi^{b}(0,1,2,5,11,18,19))
\end{equation}
 Thus the starting point is matching of plane $(0,1,2,5,11,18,19)$
to the line $(0,1,18)$ lying on it. Every other plane, obtained by
applying $b$ $(0\leq b\leq4)$ Frobenius followed by $a$ $(0\leq a\leq30)$
shift automorphisms to this plane, is matched to the line obtained
by applying the same sequence of automorphisms to line $(0,1,18)$.
Varying $a$ and $b$ fully($31\times5$ cases), we obtain complete
mapping for each of the $155$ planes in form of a \textit{perfect
matching}. Similarly, we can create $6$ other such perfect matching
patterns $S_{2},\cdots,S_{7}$ by mapping the first plane to the other
$6$ lines lying on it, one at a time.

The inverses of each of these matching patterns also form \textit{different}
perfect matching patterns from the set of lines to the set of planes.
These will be denoted by $S_{i}^{-1}$.

\section{LU/Cholesky Decomposition Algorithm Mapping Scheme - I}
\label{scheme1}

The first decomposition scheme uses the concept of
perfect matchings to design \textit{direct} interconnect network.
This scheme is a concrete realization of very brief decomposition
outline suggested by Karmarkar \,\cite{karm1}. In this scheme, we
use all the $155$ processors and memory blocks. For easy implementation,
each processor is connected to its own exclusive memory block(which
can be main memory), and this pair is associated with a line of ${\mathbb{P}}(4,GF(2))$.
In addition, the processor is also associated with the plane mapped
to the line through the perfect matching $S_{1}$. The plane signifies
the computation which is scheduled for this processor. Each such processor
is \textit{directly} connected to $12$ other processor-memory pairs
to form its interconnection subnetwork - $6$ pairs mapped to its
plane through perfect matchings $S_{2}$ to $S_{7}$, and $6$ more
through inverse perfect matchings $S_{2}^{-1}$ to $S_{7}^{-1}$.

The data of matrix $A$ is distributed among different memory blocks
in manner similar to one used in ScaLAPACK\,\cite{scalapack2}. It
is first partitioned into $B\times B$ blocks of size $b\times b$.
Each block index is then mapped to a point in geometry, obtained by
taking its residue modulo 31. Block $A_{p,q}$ is then placed in the
memory block given by points p(mod 31) and q(mod 31), using function
$M$: $\Omega_{0}\times\Omega_{0}$ $\rightarrow$ $\Omega_{1}$.
\begin{eqnarray*}
        M(i,j) & = & \mbox{line joining points i \& j } \forall i,\;j\;
        \in\; {0,\ldots,30} \mbox{~and~} i\neq j \\
M(i,i) &= & L_{x}^{i}(\Phi^{a}(0,1,18)),\;a\in{0,1,2,3,4} \mbox{(5 copies of \ensuremath{i^{th}}diagonal block)}
\end{eqnarray*}
For mapping computations, the triplet of block indices triplet is
first mapped to a triplet of points, in same way as above. It is then
mapped to a plane using the map $C_{1}:$ ($\Omega_{0}$x$\Omega_{0}$x$\Omega_{0}$)
$\rightarrow$ $\Omega_{2}$.
\begin{eqnarray*}
C_{1}(i,j,i)\: or\: C_{1}(i,i,k) & =& S_{1}^{-1}\textit{(line joining i and j or k)}\\
C_{1}(i,j,k) & = & \textit{plane through non-collinear points i, j and k}\\
C_{1}(i,j,j) & = & \textit{plane passing through i and the 5 lines M(i,j)}
\end{eqnarray*}
Our focus has been on a scheme with \textbf{improved} performance,
in which we take an approach \textbf{different} than Karmarkar's,
that we present in next section. Further, due to lack of space, we omit the details
of the computation, other than re-used portions such as matrix-block
mapping, as well as of the provably efficient communication schedule
for this scheme. Only brief remarks follow. The details may be found
in \cite{aby_rep}.

Over all iterations, this scheduling scheme makes use of all processing
units in ${\mathbb{P}}(4,GF(2))$. Therefore, perfect matching patterns
were applied here to develop communication strategies. This makes
the design scalable -- the use of a bigger geometry and remapping
of the problem is easy in this case.

However, this design under-utilizes the parallelism available: either
15 or 35 out of 155 processors are working every cycle, though they
are load-balanced among themselves. In later iterations, lesser number
of trailing matrix computations need to be parallelized, and correspondingly
lesser number of operational processors. Also, in later iterations,
because of duplication of diagonal elements, certain blocks get communicated
to extra number of processors. With these considerations, we propose
and evaluate a \textbf{novel scheme}, which improves upon the resource
usage as well as the wiring density.

\section{LU/Cholesky Decomposition Algorithm Mapping Scheme - II}

In previous scheme, each node had 12 direct connections, leading to
high wire density. A better scheme inspired by perfect difference
networks(PDN) is presented. Having node degree as 7, it saves upon
interconnection cost and complexity.

Using the same geometry ${\mathbb{P}}(4,GF(2))$, each processor is
once again paired up with one memory block, and each such pair is
associated with a line of the geometry. These pairs are then interconnected
to each other via \textbf{buses}; each bus mapped to a plane. Thus,
we have 155 buses in all and each bus is connected to $7$ processors
corresponding to the lines that are incident on its representative
plane.

\subsection{Motivating the Scheme}
To motivate the logic behind using multiple buses, we illustrate the
update on $(j,k){}^{\text{th}}$ block during $0^{\text{th}}$ phase.
We would like to do all these updates in parallel. The charge of block
$A{}_{\text{j,k}}$ is with the processor associated with the line
through points j and k. Thus, each processor is typically in charge
of 6 blocks. Recall from section \ref{lu_ov} that the update of block
$A{}_{\text{j,k}}$ during $0^{\text{th}}$ phase requires $L{}_{\text{j,0}}$,
$U{}_{\text{0,k}}$ (normalized blocks of $0^{\text{th}}$ column
and row respectively). These blocks are with the processors whose
line contains point 0. Thus, processor of line through {[}j, k{]}
must have connection with those processors, which map to lines through
{[}j, 0{]} and {[}0, k{]}, so that $L{}_{\text{j,0}}$, $U{}_{\text{0,k}}$
can be communicated. If points 0, j, k are \textit{collinear}, then
this communication is local, i.e. to the processor's local memory.
Else, there exists a unique plane through points 0, j, k. This plane
would contain the above three lines (totally 7 lines per plane). Hence
the requirement is to have interconnections between 7 processors associated
with 7 lines of a plane, which can be naturally be implemented using
a \underline{BUS}. The processors with point 0 in their line, must
broadcast and the other processors (without point 0) must ''listen''.
Each listening processor, listens to the \textbf{unique} BUS (corresponding
to the plane containing the line of the processor and point 0). For
each plane, at a given time, at most one processor may broadcast.
After fixed number of communication steps, each of the 155 processors
will be able to ''compute'' its trailing matrix update concurrently.

\subsection{Data Distribution}
The \textit{distribution of data} among memory blocks is identical
to the one in previous scheme. Every block $A_{i,j}$ with distinct
$i$, $j$ gets stored in the memory module of processor/line $(i,j)$.
Each diagonal block $A_{i,i}$ is stored in 5 memory modules. This
duplication helps in fast communication, as discussed later. The \textit{distribution
of computational load} (characterized by triplets described earlier)
is done using the following modified function $C_{2}$.
\begin{eqnarray*}
C_{2}: \{C_{2}(i,j,k) \mbox{\hspace{42pt}}& = & \textit{line through points j \&
k}, j \neq k\} \wedge \\
\{C_{2}(i,j,j) \mbox{\hspace{42pt}}& = & \textit{lines to which
$A_{j,j}$ is allocated}\} \wedge \\
\{\textit{$C_{2}$(i,j,i) or $C_{2}$(i,i,k)}
& = & \textit{line joining (i and j) or (i and k)}\}
\end{eqnarray*}

\begin{figure*}
\begin{centering}
\includegraphics[scale=0.49]{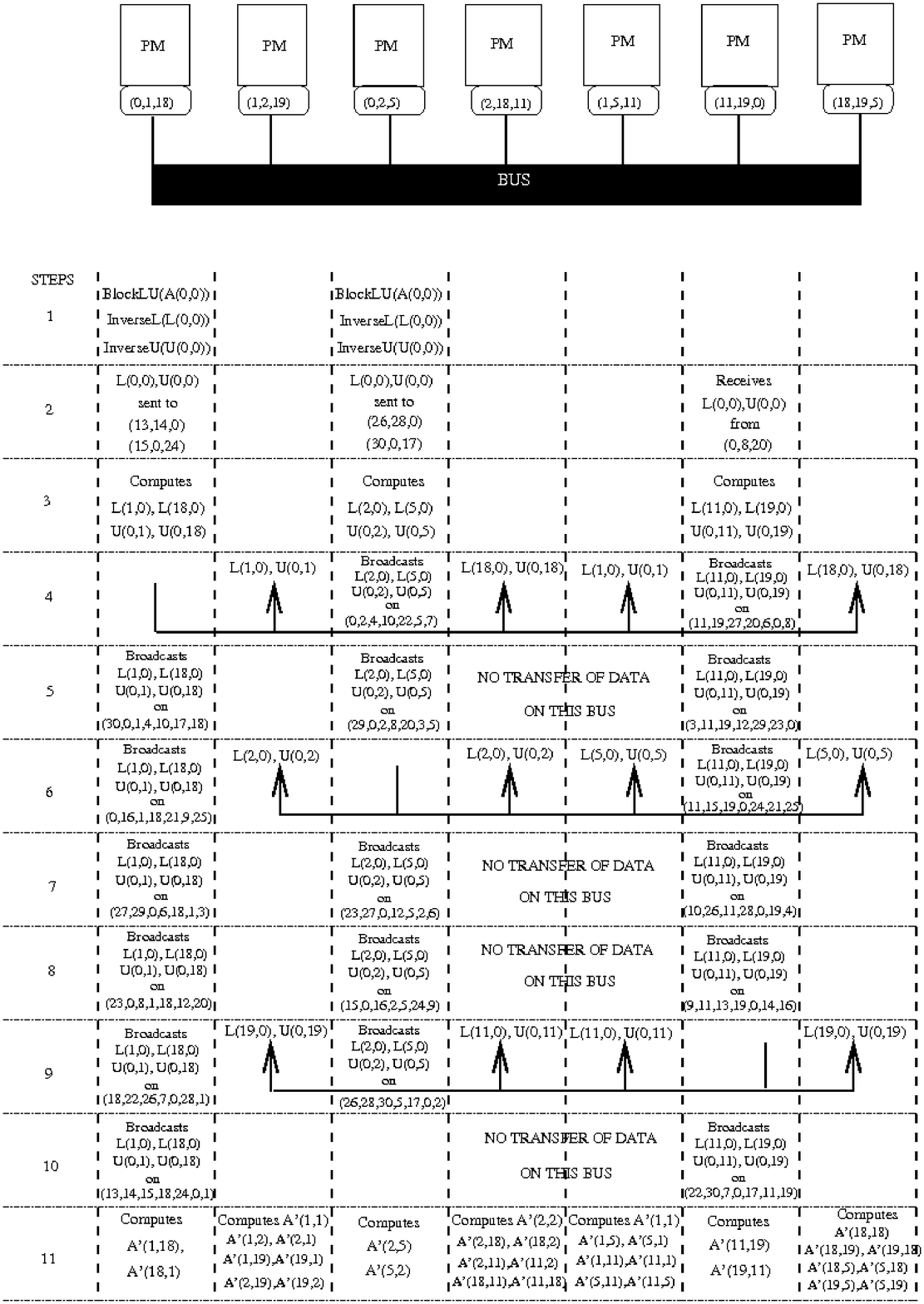}
\end{centering}
\caption{Scheme II: Execution of $0^{th}$ iteration on bus (0, 1, 2, 5, 11,
18, 19)}
\label{sch2}
\end{figure*}
\subsection{Illustration of Mapping}
We provide a running example of $i^{th}$ iteration to illustrate
the mapping. For illustration, the complete $0^{th}$ iteration is
depicted in figure \ref{sch2}. The first step, block LU decomposition
and its inverse, is simultaneously computed on $5$ different processors
per block, like in $1^{st}$ scheme. For usage in row and column updates,
data from these 5 processors needs to be transferred to all processors,
which will perform these updates in $i^{th}$ iteration. A point(represented
by $i$) lies on 15 lines in ${\mathbb{P}}(4,GF(2))$, and hence 15
processors perform these updates per iteration. Hence blocks $L_{i,i}^{-1}$
and $U_{i,i}^{-1}$ are transmitted on 5 buses, mapped to 5 processors
having these blocks. Buses are chosen such that the 15 processors
can receive them in conflict-free way. In ${\mathbb{P}}(4,GF(2))$,
each line is contained in 7 planes, and hence each processor can potentially
communicate with 7 buses. At each cycle, one processor on the bus
transmits on the bus, while remaining may or may not read the data
from the bus. By carefully choosing the 5 processors in $1^{st}$
step(starting with a line and taking its Frobenius automorphisms),
and selecting 5 suitable buses, \textit{matrix data $L_{i,i}^{-1}$
and $U_{i,i}^{-1}$ can be distributed in just one cycle without any
conflict}. \underline{Duplicating $1^{st}$ step on 5 processors thus
saves many communication cycles}. As an example, in $0^{th}$ iteration,
the 5 processors storing $A_{0,0}$ perform LU decomposition followed
by computation of their inverses. These inverses are then transferred
to other processors associated with $0$-containing lines through
$5$ buses. Processor $(0,1,18)$ transmits on bus $(13,14,15,18,24,0,1)$,
$(0,2,5)$ on $(26,28,30,5,17,0,2)$, $(0,4,10)$ on $(21,25,29,10,3,0,4)$,
$(0,8,20)$ on $(11,19,27,20,6,0,8)$, and $(0,16,9)$ on $(22,7,23,9,12,0,16)$.
The 5 processor nodes are linked via \textit{Frobenius automorphisms}.
For the $i^{th}$ iteration, each of the above lines and planes are
shifted $i$ times, using \textit{Shift automorphisms}. The \textit{row/column
updates} for $0^{th}$ iteration are scheduled now on each of the
$15$ processors with blocks $A_{0,j}$ and $A_{j,0}$, $j\in{1,\ldots,(B-1)}$,
where $B$ is number of blocks.

The \textit{trailing matrix update} for the $(j,k)^{th}$ block is
performed by processor corresponding to line $(j,k)$. This step requires
$L_{j,i}$ and $U_{i,k}$, calculated by some other processors during
row/column update. Hence these blocks need to be moved in to processor
$(j,k)$. This communication is done in $\mathbf{7}$ steps. Each
processor broadcasts these blocks on each of the 7 buses connected
to it. In $q^{th}$ step, the processor having these blocks broadcasts
it on the bus mapped to it through perfect matching pattern $S_{q}^{-1}$,
$q\in{1,\ldots,7}$, mentioned earlier. The perfect matching pattern
\textit{ensures} that in each cycle, a bus is controlled by 1 processor
only, and that a processor receives data from only 1 bus in one cycle.
By the end of these steps, data has been broadcast once on every bus
connected to a particular processor. E.g., in $i^{th}$ iteration,
a processor represented by line $(x,y,x+y)$ will need $L_{m,i}$
and $U_{i,m}$, if one of its indices is same as $m$. In such a case,
the line corresponding to this processor and the line corresponding
to processor containing $L_{m,i}$ or $U_{i,m}$ share a common point,
and hence a bus represented by some plane. So at some stage within
the 7 cycles, this processor will get the required data on the shared
bus. Thus, at the end of these $7$ steps, each processor will have
the entire information needed by it to calculate trailing matrix updates.
All the 155 processors now simultaneously compute the updates for
the trailing matrix blocks that they possess, thus completing the
$i^{th}$ iteration.

\subsection{Coherency and Synchronization Issues}

Coherence is a primary design issue in multiprocessing scenario. Since
we need only one level of memory hierarchy, coherence issue can (only)
arise in context of 5 copies of diagonal blocks in 5 processors, during
each iteration. These data copies are to be used in later cycles of
the iteration, where reading of different, stale data can be a potential
issue. However, these diagonal blocks as well as $L_{i,i}$, $U_{i,i}$
are all identical in each iteration, and hence there is no incoherence
of data.

Data synchronization conflicts arise when two or more processors try
to access some data object at the same time. Given that at a particular
moment, all active processors are doing identical computation, the
only possible conflict is a write-write conflict for a particular
matrix block in a given cycle, during trailing update. However, only
one processor works with a particular block in a cycle \textit{in
a particular iteration}, and hence synchronization conflicts also
don't arise in this scheme. Further, none of the processors involved
in sending data has to receive any data in same cycle, so there is
no I/O conflict.

\subsection{Design Analysis}

In this scheme, in most iterations, the parallelism provided by 155
processing elements is used up, bettering the degree of parallelism
exploited by the earlier scheme-I. Hence this design is scalable.
Also, the distribution of computation is almost balanced. The source
of imbalance is at later stages, when the size of matrix to be updated
``trails off'' leading to a cut in required amount
of parallelism. The resource usage is high, which leads to better
time performance. The communication is also lesser than previous scheme,
due to only L and U blocks being transferred during each iteration.
Broadcasting per processor on each bus takes one cycle, hence there
is no scope of
changing bus data transfer mode to e.g. split transactions, which can boost
the throughput further.

Having a possible target of distributed embedded systems, the
buses are expected to be off-chip/backplane buses. Traditionally,
shared buses for off-chip purposes have been implemented using tri-state
buses that drive bidirectional lines. The advantage of tri-state
bidirectional buses is that they take up fewer wires and have a smaller
area footprint. This is important since we use many such buses, and hence
their resource requirements need to be minimum. One issue is that due to
large number of buses involved, there could be more power
consumption. However, in system design, power consumption in general trades
off with throughput, and similarly we gain throughput here while consuming
some more power. Alternate interconnection patterns such as design of novel
switches to do simultaneous communication can be considered to alleviate
this problem.

\section{Experimental Results for LU Decomposition}

C++ programs for a mesh-based scheme(for comparison) and the 2 PG-based
schemes were developed for both correctness and performance evaluation.
By simulations on a uniprocessing system, the three schemes were indeed
found to be working correctly. Also, the performance figures calculated
analytically for these three schemes have been tested to be correct
by instrumenting the corresponding programs. The sources of the programs
are available with the authors on request. The tabulated experimental
data is presented in table \ref{table:comp}. For analysis, the active
period of each processor is classified into three categories. A processor
either spends time doing O($b^{3}$) computations, or O($b^{2}$)
computations, or O($b^{2}$) communication. The O($b^{3}$) computations
comprise of block LU decompositions, matrix inversions and matrix
multiplications. The O($b^{2}$) computations comprise of matrix subtraction
done during trailing update. The table also contains the \textit{average}
number of cycles in which each processor is active for the 3 different
categories.\textit{ Average processor utilization} in each category
can then be defined as the average number of cycles taken by processor
in that category, divided by total number of cycles in which all processors
finish that category's job in parallel. The normalized computation
times are reported too.

\begin{table*}[h]
\caption{Cycle Counts for Various Schemes}
\label{table:comp}
\centering
\begin{tabular}{|c|c|c|c|c|c|c|c|c|c|}  \hline
{\footnotesize Scheme/ } & \multicolumn{3}{|c|}{{\footnotesize Total Cycles}} & \multicolumn{3}{|c|}{{\footnotesize Time required}} & \multicolumn{3}{|c|}{{\footnotesize Average Cycles}}\tabularnewline \hline
{\footnotesize Block } & {\footnotesize Comp } & {\footnotesize Comm } & {\footnotesize Sub } & {\footnotesize Comp } & {\footnotesize Comm } & {\footnotesize Sub } & {\footnotesize Comp } & {\footnotesize Comm } & {\footnotesize Sub }\tabularnewline
{\footnotesize Size } & {\footnotesize O($b^{3}$) } & {\footnotesize O($b^{2}$) } & {\footnotesize O($b^{2}$) } & {\footnotesize O($b^{3}$) } & {\footnotesize O($b^{2}$) } & {\footnotesize O($b^{2}$) } & {\footnotesize O($b^{3}$) } & {\footnotesize O($b^{2}$) } & {\footnotesize O($b^{2}$) }\tabularnewline
\hline \hline
{\footnotesize Mesh:62$\times$62 } & {\footnotesize 69 } & {\footnotesize 221 } & {\footnotesize 11 } & {\footnotesize 1189.56 } & {\footnotesize 61.35 } & {\footnotesize 3.05 } & {\footnotesize 4 } & {\footnotesize 7 } & {\footnotesize 3 }\tabularnewline \hline
{\footnotesize Mesh:31$\times$31 } & {\footnotesize 196 } & {\footnotesize 583 } & {\footnotesize 56 } & {\footnotesize 422.38 } & {\footnotesize 40.46 } & {\footnotesize 3.89 } & {\footnotesize 34 } & {\footnotesize 41 } & {\footnotesize 30 }\tabularnewline \hline
{\footnotesize PG-1:24$\times$24 } & {\footnotesize 856 } & {\footnotesize 2241 } & {\footnotesize 651 } & {\footnotesize 856.00 } & {\footnotesize 93.23 } & {\footnotesize 27.08 } & {\footnotesize 82 } & {\footnotesize 171 } & {\footnotesize 73 }\tabularnewline \hline
{\footnotesize PG-1:12$\times$12 } & {\footnotesize 5023 } & {\footnotesize 12635 } & {\footnotesize 4427 } & {\footnotesize 627.88 } & {\footnotesize 131.40 } & {\footnotesize 46.04 } & {\footnotesize 669 } & {\footnotesize 1234 } & {\footnotesize 633}\tabularnewline \hline
{\footnotesize PG-1:8$\times$8 } & {\footnotesize 15291 } & {\footnotesize 37981 } & {\footnotesize 14118 } & {\footnotesize 565.77 } & {\footnotesize 175.47 } & {\footnotesize 65.23 } & {\footnotesize 2319 } & {\footnotesize 4050 } & {\footnotesize 2238}\tabularnewline \hline
{\footnotesize PG-2:24$\times$24 } & {\footnotesize 393 } & {\footnotesize 846 } & {\footnotesize 188 } & {\footnotesize 393.00 } & {\footnotesize 35.19 } & {\footnotesize 7.82 } & {\footnotesize 82 } & {\footnotesize 44 } & {\footnotesize 73 }\tabularnewline \hline
{\footnotesize PG-2:12$\times$12 } & {\footnotesize 1693 } & {\footnotesize 2994 } & {\footnotesize 1097 } & {\footnotesize 211.63 } & {\footnotesize 31.14 } & {\footnotesize 11.41 } & {\footnotesize 669 } & {\footnotesize 214 } & {\footnotesize 633 }\tabularnewline \hline
{\footnotesize PG-2:8$\times$8 } & {\footnotesize 4458 } & {\footnotesize 6444 } & {\footnotesize 3285 } & {\footnotesize 164.95 } & {\footnotesize 29.77 } & {\footnotesize 15.18 } & {\footnotesize 2319 } & {\footnotesize 510 } & {\footnotesize 2238 }\tabularnewline \hline
\end{tabular}
\end{table*}

An important observation was that across all schemes, as block size
decreases, the performance improves almost linearly due to more fine-grained
distribution of computational load. Another observation was that our
schemes have much better processor utilization than mesh scheme
for all different block sizes. Between the two schemes, for each given
block size, the $2^{nd}$ scheme needs much lesser amount of average
communication cycles while having same average computation cycles,
and hence improves upon the $1^{st}$ scheme.

In terms of implementation complexity, our architecture is easy to
implement for medium-sized matrices with cheap uniprocessors having
one level of cache, or just main memory. Each processor will need
to store approximately $\frac{n^{2}}{155}$ matrix elements, which
for medium-sized matrices, can fit in its (main) memory or even L1
caches. Like cluster computing, we can use off-the-shelf cheap processors
to interconnect and set up the configuration required for these computation
schemes. However, unlike cluster computing, these schemes use special
interconnects described earlier, not LAN. Given the low complexity
of individual processors, a customized board with multiple lightweight
microprocessor IP cores can be designed to significantly reduce the
form factor of such system.

The main aim of these simulations was to validate the correctness
of the communication and computation schedules based on projective
geometry. It may be remarked that, as indicated in\,\cite{rakov2},
the projective geometry like architectures can provide good implementations
of important communication primitives in distributed high performance
computing.

\section{Conclusion and Future Work}

For CG algorithm, although the performance difference trend between the
two distributions is slightly unclear till about 20-30 processes, the
percentage improvements seen beyond this cluster size show that the
projective
distribution seems to be performing better than row-wise distribution.
In future, we have plans to (a) exploit the symmetry exposed in the
projective distribution, (b) apply the projective distribution in multi/many core
configurations as well as GPGPU configurations,  (c) extend
these ideas to more general sparse matrices, as well as
(d) study other characteristics (e.g. load-balancing behaviour etc.) of
projective distribution.

For LU/Cholesky decomposition, we have introduced two new schemes
for processor interconnection based on projective geometry graphs,
which work efficiently. Results show that in terms of processor utilization
and total time required, these schemes do better than the conventional
mesh-based scheme. One direction of research is to make these schemes
handle large-sized matrices. Possible interesting implementations
include using distributed shared memory schemes over 1-level memory
hierarchy to accommodate storage of bigger blocks. This work suggests
suitability of using projective-geometry based graphs for application-specific
system design, in contrast to generic design setup in \,\cite{rakov1}.
As a matter of fact, variations of topologies derived from PG-based
graphs that are succinct for other potential parallel computations
have also yielded promising results, especially in the area of \textit{error
correction coding}\,\cite{expanders} and digital systems design\,
\cite{cacs_pap}.

\section{Acknowledgements}
The work was carried out with support from Innovation Labs, Tata
Consultancy Services Ltd, Bangalore, under project ID 1009295.
The authors will like to thank Dr Sunil Sherlekar, and Dr Rajendra
Lagu for the numerous discussions and suggestions related to this work.

\bibliography{ref}

\end{document}